\documentclass[seceqn,secthm,9pts]{article}
\usepackage{amsfonts}
\usepackage{mitpress}
\usepackage{amsmath}
\usepackage{amssymb}
\usepackage{graphicx}

\newtheorem{theorem}{Theorem}

\newtheorem{corollary}[theorem]{Corollary}

\newenvironment{proof}[1][Proof]{\noindent \textbf{#1.} }{\  \rule{0.5em}{0.5em}}
\newdimen \dummy
\dummy=\oddsidemargin
\input{tcilatex}
\begin{document}

\title{Incomplete Riemann-Liouville fractional derivative operators and
incomplete hypergeometric functions}
\author{Mehmet Ali \"{O}zarslan and Ceren Ustao\u{g}lu \\
Eastern Mediterranean University\\
Gazimagusa, TRNC, Mersin 10, Turkey \\
Email: mehmetali.ozarslan@emu.edu.tr, ceren.ustaoglu@emu.edu.tr}
\maketitle

\begin{abstract}
In this paper, the incomplete Pochhammer ratios are defined in terms of the
incomplete beta function $B_{y}(x,z)$. With the help of these incomplete
Pochhammer ratios we introduce new incomplete Gauss, confluent
hypergeometric and Appell's functions and investigate several properties of
them such as integral representations, derivative formulas, transformation
formulas and recurrence relation. Furthermore, an incomplete
Riemann-Liouville fractional derivative operators are introduced. This
definition helps us to obtain linear and bilinear generating relations for
the new incomplete Gauss hypergeometric functions.
\end{abstract}


\section{\protect\bigskip Introduction}

In recent years , some extensions of the well known special functions have
been considered by several authors (see, for example, \cite{3}, \cite{4}, 
\cite{6}, \cite{15}, \cite{16}, \cite{17}, \cite{18}, \cite{19}, \cite{29}
). The familiar incomplete gamma fuctions $\gamma (s,x)$ and $\Gamma (s,x)$
are defined by 
\begin{equation*}
\gamma (s,x):=\int_{0}^{x}t^{s-1}e^{-t}dt\text{ \ \ \ \ \ \ \ }\left( \func{%
Re}(s)>0;\text{ }x\geqq 0\right)
\end{equation*}%
and 
\begin{equation*}
\Gamma (s,x):=\int_{x}^{\infty }t^{s-1}e^{-t}dt\text{ \ \ \ \ }(x\geqq 0;%
\text{ }\func{Re}(s)>0\text{ when }x=0),
\end{equation*}%
respectively. They satisfy the following decomposition formula:%
\begin{equation}
\gamma (s,x)+\Gamma (s,x)=\Gamma (s)\text{ \ \ }\left( \func{Re}(s)>0\right)
.  \label{1.3}
\end{equation}

The function $\Gamma (s)$ and its incomplete versions $\gamma (s,x)$ and $%
\Gamma (s,x),$ play important roles in the study of analytical solutions of
a variety of problems in diverse areas of science and engineering \cite{14}.

The widely used Pochhammer symbol $\left( \lambda \right) _{\nu }$ $\left(
\lambda ,\nu \in 
\mathbb{C}
\right) $ is defined, in general, by%
\begin{equation}
\left( \lambda \right) _{\nu }:=\frac{\Gamma \left( \lambda +\nu \right) }{%
\Gamma \left( \lambda \right) }=\left\{ 
\begin{array}{ll}
1 & \left( \nu =0;\text{ }\lambda \in 
\mathbb{C}
\backslash \left\{ 0\right\} \right) \\ 
\lambda \left( \lambda +1\right) ...\left( \lambda +\nu -1\right) \text{ \ \
\ } & \left( \nu \in 
\mathbb{N}
;\text{ }\lambda \in 
\mathbb{C}
\right)%
\end{array}%
\right\}  \label{1.4}
\end{equation}

In terms of the incomplete gamma functions $\gamma (s,x)$ and $\Gamma (s,x)$%
, the incomplete Pochhammer symbols $\left( \lambda ;x\right) _{\nu }$ and $%
\left[ \lambda ;x\right] _{\nu }$ $\left( \lambda ;\nu \in 
\mathbb{C}
;\text{ }x\geqq 0\right) $ were defined as follows \cite{5}:%
\begin{equation}
\left( \lambda ;x\right) _{\nu }:=\frac{\gamma (\lambda +\nu ,x)}{\Gamma
\left( \lambda \right) }\text{ \ \ \ \ }\left( \lambda ,\nu \in 
\mathbb{C}
;\text{ }x\geqq 0\right)  \label{1.5}
\end{equation}%
and%
\begin{equation}
\left[ \lambda ;x\right] _{\nu }:=\frac{\Gamma \left( \lambda +\nu ,x\right) 
}{\Gamma \left( \lambda \right) }\text{ \ \ \ }\left( \lambda ,\nu \in 
\mathbb{C}
;\text{ }x\geqq 0\right) .  \label{1.6}
\end{equation}

In view of (\ref{1.3}), these incomplete Pochhammer symbols $\left( \lambda
;x\right) _{\nu }$ and $\left[ \lambda ;x\right] _{\nu }$ satisfy the
following decomposition relation:%
\begin{equation}
\left( \lambda ;x\right) _{\nu }+\left[ \lambda ;x\right] _{\nu }=\left(
\lambda \right) _{\nu }\text{ \ \ }\left( \lambda ,\nu \in 
\mathbb{C}
;\text{ }x\geqq 0\right) ,  \label{1.7}
\end{equation}%
where $\left( \lambda \right) _{\nu }$ is the Pochhammer symbol given by (%
\ref{1.4}).

The incomplete Gauss hypergeometric functions were defined by means of the
incomplete gamma functions as follows \cite{5}:%
\begin{equation}
_{2}\gamma _{1}\left[ 
\begin{array}{cccc}
\left( a,x\right) & . & b & ; \\ 
&  & c & ;%
\end{array}%
\begin{array}{c}
z%
\end{array}%
\right] :=\dsum\limits_{n=0}^{\infty }\frac{(a;x)_{n}(b)_{n}}{(c)_{n}}\frac{%
z^{n}}{n!}  \label{1.11}
\end{equation}%
and%
\begin{equation}
_{2}\Gamma _{1}\left[ 
\begin{array}{cccc}
(a,x) & . & b & ; \\ 
&  & c & ;%
\end{array}%
\begin{array}{c}
z%
\end{array}%
\right] :=\dsum\limits_{n=0}^{\infty }\frac{[a;x]_{n}(b)_{n}}{(c)_{n}}\frac{%
z^{n}}{n!}.  \label{1.12}
\end{equation}

In view of (\ref{1.3}), these incomplete Gauss hypergeometric functions are
satisfy the following decomposition relation:%
\begin{equation}
_{2}\gamma _{1}\left[ 
\begin{array}{cccc}
\left( a,x\right) & . & b & ; \\ 
&  & c & ;%
\end{array}%
\begin{array}{c}
z%
\end{array}%
\right] +~_{2}\Gamma _{1}\left[ 
\begin{array}{cccc}
(a,x) & . & b & ; \\ 
&  & c & ;%
\end{array}%
\begin{array}{c}
z%
\end{array}%
\right] =~_{2}F_{1}\left[ 
\begin{array}{cccc}
a & . & b & ; \\ 
&  & c & ;%
\end{array}%
\begin{array}{c}
z%
\end{array}%
\right] .  \label{1.13}
\end{equation}

It should be also mentioned that, for all that, H. M. Srivastava, M. Aslam
Chaudhry and Ravi P. Agarwal discussed some properties and some interesting
applications of these families of incomplete hypergeometric functions \cite%
{24}. In recent years, the gamma function $\Gamma \left( z\right) $ and the
Pochhammer symbol $\left( \lambda \right) _{v}$ were used to extend the
generalized hypergeometric functions and their multivariate versions. After
this works, incomplete hypergeometric functions have become one of the hot
topics of recent years \cite{1}, \cite{6}, \cite{7}, \cite{8}, \cite{9}, 
\cite{15}, \cite{20}, \cite{21}, \cite{22}, \cite{23}, \cite{26}, \cite{27}, 
\cite{28}, \cite{29}.

On the other hand, fractional derivative operators found applications in
many diverse areas of mathematical, physical and engineering problems.
Because of this reason these operators have been an active research in
recent years \cite{2}, \cite{10}, \cite{11}, \cite{12}, \cite{13}, \cite{14}%
, \cite{30}. The use of fractional derivative operators in obtaining
generating relations for some special functionscan be found in \cite{16}, 
\cite{25}. In the present paper, we are aimed to introduce new incomplete
hypergeometric functions with the aid of incomplete Pochhammer ratios and
investigate their certain properties. Moreover, we introduce incomplete
Riemann-Liouville fractional derivative operators and we obtain some
generating relations for these new incomplete hypergeometric function with
the aid of these new defined operators. The organization of the paper as
follows:

In Section 2, the incomplete Pochhammer ratios are introduced by using the
incomplete beta function and some derivative formulas to involving these new
incomplete Pochhammer ratios are investigated. In Section 3, new incomplete
Gauss hypergeometric functions and confluent hypergeometric functions are
introduced with the help of these incomplete Pochhammer ratios and integral
representations, derivative formulas, transformation formulas and recurrence
relation are obtained for them. In Section 4, we define new incomplete
Appell's functions $F_{1}[a,b,c;d;x,z;y]$, $F_{1}\{a,b,c;d;x,z;y\}$, $%
F_{2}[a,b,c;d,e;x,z;y]$ and $F_{2}\{a,b,c;d,e;x,z;y\}$ and obtain their
integral representations. In Section 5, we introduce incomplete Riemann-
Liouville fractional derivative operator and show that the incomplete
Riemann-Liouville fractional derivative of some elementary fuctions give the
new incomplete fuctions defined in Sections 3 and 4. Finally, in the last
section, we obtain linear and bilinear generating relations for incomplete
hypergeometric functions.

\section{The incomplete Pochhammer Ratio}

The incomplete beta function is defined by%
\begin{equation}
B_{y}(x,z):=\int_{0}^{y}t^{x-1}(1-t)^{z-1}dt,\ \func{Re}(x)>\func{Re}%
(z)>0,~0\leq y<1  \label{2.1}
\end{equation}%
and can be expressed in terms of the Gauss hypergeometric function 
\begin{equation}
B_{y}(x,z):=\frac{y^{x}}{x}(1-y)^{z}~_{2}F_{1}(1,x+z;1+x;y).  \label{2.2}
\end{equation}%
The incomplete beta function satisfy the following relation:%
\begin{equation}
B_{y}(b+n,c-b)+B_{1-y}(c-b,b+n)=B(b+n,c-b),~~~n\in 
\mathbb{N}
_{0}:=%
\mathbb{N}
\cup \left\{ 0\right\} .  \label{2.3}
\end{equation}

In terms of the incomplete beta function $B_{y}(x,z),$ the incomplete
Pochhammer ratios $\left[ b,c;y\right] _{n}$ and $\left\{ b,c;y\right\} _{n}$
are introduced as follows:

\begin{equation}
\left[ b,c;y\right] _{n}:=\frac{B_{y}\left( b+n,c-b\right) }{B\left(
b,c-b\right) }  \label{2.4}
\end{equation}%
and

\begin{equation}
\left\{ b,c;y\right\} _{n}:=\frac{B_{1-y}(c-b,b+n)}{B(b,c-b)}\text{ }
\label{2.5}
\end{equation}%
where $0\leq y<1.$ It is clear from (\ref{2.3}) that%
\begin{equation}
\left[ b,c;y\right] _{n}+\left\{ b,c;y\right\} _{n}=\frac{\left( b\right)
_{n}}{\left( c\right) _{n}}.  \label{5.22}
\end{equation}

In view of (\ref{2.2}) , we have the following relations 
\begin{equation}
\left[ b,c;y\right] _{n}:=\frac{1}{B(b,c-b)}\frac{y^{b+n}}{b+n}%
(1-y)^{c-b}{}_{2}F_{1}(1,c+n;b+n+1;y)  \label{2.6}
\end{equation}%
and%
\begin{equation}
\left\{ b,c;y\right\} _{n}:=\frac{1}{B(b,c-b)}\frac{(1-y)^{c-b}}{c-b}%
y^{b+n}{}_{2}F_{1}(1,c+n;1+c-b;1-y).  \label{5.7}
\end{equation}

In the following theorem, we investigate the $n-th~$derivatives of the
incomplete beta fuction by means of incomplete Pochhammer ratios.

\begin{theorem}
The following derivative formulas hold true:%
\begin{equation}
\lbrack b,c;y]_{n}=\frac{\left( -1\right) ^{n}\Gamma \left( c\right) }{%
\Gamma \left( c-b+n\right) \Gamma \left( b\right) }y^{b+n}\frac{d^{n}}{dy^{n}%
}\left[ y^{-b}B_{y}(b,c-b+n)\right] ,~  \label{2.7}
\end{equation}%
and%
\begin{equation}
\{b,c;y\}_{n}=\frac{\Gamma (b+n)}{\Gamma (b+2n)}\frac{1}{B(b,c-b)}(1-y)^{c-b}%
\frac{d^{n}}{dy^{n}}((1-y)^{-c+b+n}B_{1-y}(c-b-n,b+2n)).  \label{2.9}
\end{equation}
\end{theorem}

\begin{proof}
Using (\ref{2.1}) and (\ref{2.4}), we immediately obtain the following
equation:%
\begin{equation*}
\lbrack b,c;y]_{n}=\frac{y^{b+n}}{B(b,c-b)}%
\int_{0}^{1}u^{b+n-1}(1-uy)^{c-b-1}du.
\end{equation*}%
On the other hand, we have%
\begin{equation}
y^{-b}B_{y}(b,c-b+n)=\int_{0}^{1}u^{b-1}(1-uy)^{c-b+n-1}du.  \label{2.8}
\end{equation}%
Taking derivatives $n$ times on both sides of (\ref{2.8}) with respect to $y$%
, we can obtain a derivative formula for the incomplete beta function $%
[b,c;y]_{n}$ asserted by (\ref{2.7}). Formula (\ref{2.9}) can be proved in a
similar way.
\end{proof}

\section{The new incomplete Gauss and confluent hypergeometric functions}

In this section, we introduce new incomplete Gauss and confluent
hypergeometric functions by%
\begin{equation}
_{2}F_{1}(a,[b,c;y];x):=\sum_{n=0}^{\infty }(a)_{n}[b,c;y]_{n}\frac{x^{n}}{n!%
},  \label{2.10}
\end{equation}%
\ \ \ \ \ \ \ \ 
\begin{equation}
_{2}F_{1}(a,\left\{ b,c;y\right\} ;x):=\sum_{n=0}^{\infty }(a)_{n}\left\{
b,c;y\right\} _{n}\frac{x^{n}}{n!},  \label{2.11}
\end{equation}%
\begin{equation}
_{1}F_{1}([a,b;y];x):=\sum_{n=0}^{\infty }[a,b;y]_{n}\frac{x^{n}}{n!},
\label{5.8}
\end{equation}%
and%
\begin{equation}
_{1}F_{1}(\left\{ a,b;y\right\} ;x):=\sum_{n=0}^{\infty }\left\{
a,b;y\right\} _{n}\frac{x^{n}}{n!}  \label{5.9}
\end{equation}%
where $0\leq y<1.$

An immediate consequence of (\ref{5.22}) and the definitions (\ref{2.10}), (%
\ref{2.11}), (\ref{5.8}) and (\ref{5.9}) are the following decomposition
formulas%
\begin{equation}
_{2}F_{1}(a,[b,c;y];x)+~_{2}F_{1}(a,\left\{ b,c;y\right\}
;x)=~_{2}F_{1}(a,b;c;x)  \label{2.12}
\end{equation}%
and%
\begin{equation}
_{1}F_{1}([a,b;y];x)+~_{1}F_{1}(\left\{ a,b;y\right\} ;x)=~_{1}F_{1}(a;b;x).
\label{5.10}
\end{equation}

\begin{theorem}
The following integral representation holds true: 
\begin{eqnarray}
_{2}F_{1}(a,[b,c;y],x) &=&\frac{y^{b}}{B(b,c-b)}%
\int_{0}^{1}u^{b-1}(1-uy)^{c-b-1}(1-xuy)^{-a}du,  \label{2.13} \\
\text{ \ \ }\func{Re}(c) &>&\func{Re}(b)>0,\text{ }\left\vert \arg
(1-x)\right\vert <\pi ).  \notag
\end{eqnarray}
\end{theorem}

\begin{proof}
Replacing the incomplete Pochhammer ratio $[b,c;y]$ in the definition (\ref%
{2.10}) by its integral representation given by (\ref{2.1}) and
interchanging the order of summation and integral which is permissible under
the conditions given in the hypothesis of the Theorem, we find 
\begin{equation}
_{2}F_{1}(a,[b,c;y],x)=\frac{1}{B(b,c-b)}%
\int_{0}^{y}t^{b-1}(1-t)^{c-b-1}(1-xt)^{-a}dt,  \label{5.11}
\end{equation}%
which can be written as follows:%
\begin{equation}
_{2}F_{1}(a,[b,c;y],x)=\frac{y^{b}}{B(b,c-b)}%
\int_{0}^{1}u^{b-1}(1-uy)^{c-b-1}(1-xuy)^{-a}du.  \label{2.15}
\end{equation}
\end{proof}

In a similar way, we have the following theorem:

\begin{theorem}
The following integral representation holds true: 
\begin{eqnarray}
_{2}F_{1}(a,\left\{ b,c;y\right\} ,x) &=&\frac{(1-y)^{c-b}}{B(b,c-b)}%
\int_{0}^{1}u^{c-b-1}(1-u(1-y))^{b-1}(1-x+xu(1-y))^{-a}du,\text{\ }
\label{2.16} \\
\text{\ }\func{Re}(c) &>&\func{Re}(b)>0,\left\vert \arg (1-x)\right\vert
<\pi .  \notag
\end{eqnarray}
\end{theorem}

\begin{theorem}
The following result holds true:%
\begin{equation}
_{2}F_{1}(a,[b,c;y],1)=\frac{\Gamma (c)\Gamma (c-a-b)}{\Gamma (c-a)\Gamma
(c-b)}-\frac{(1-y)^{c-b-a}y^{b}}{B(b,c-b)(c-a-b)}%
~_{2}F_{1}(c-a,1;1+c-b-a;1-y).  \label{2.26}
\end{equation}
\end{theorem}

\begin{proof}
Putting $x=1$ in (\ref{2.12}), we obtain%
\begin{eqnarray}
_{2}F_{1}(a,[b,c;y],1) &=&~_{2}F_{1}(a,b;c;1)-~_{2}F_{1}(a,\left\{
b,c;1-y\right\} ,1)  \label{2.27} \\
&=&\frac{\Gamma (c)\Gamma (c-a-b)}{\Gamma (c-a)\Gamma (c-b)}-\frac{%
(1-y)^{c-b-a}}{B(b,c-b)}\int_{0}^{1}u^{c-b-a-1}(1-u(1-y))^{b-1}du.  \notag
\end{eqnarray}%
Using the Euler's integral representation for (\ref{2.27}), we have%
\begin{equation}
_{2}F_{1}(a,[b,c;y],1)=\frac{\Gamma (c)\Gamma (c-a-b)}{\Gamma (c-a)\Gamma
(c-b)}-\frac{(1-y)^{c-b-a}}{B(b,c-b)(c-b-a)}%
~_{2}F_{1}(1-b,c-b-a;1+c-b-a;1-y).  \label{2.28}
\end{equation}%
Using transformation formula%
\begin{equation}
_{2}F_{1}(\alpha ,\beta ;\gamma ;z)=(1-z)^{\gamma -\beta -\alpha
}~_{2}F_{1}(\gamma -\alpha ,\gamma -\beta ;\gamma ;z),  \label{2.29}
\end{equation}%
in (\ref{2.28}), we obtain%
\begin{equation}
_{2}F_{1}(1-b,c-b-a;1+c-b-a;1-y)=y^{b}~_{2}F_{1}(c-a,1;1+c-b-a;1-y).
\label{2.30}
\end{equation}%
Considering (\ref{2.30}) in (\ref{2.28}), we get 
\begin{equation}
_{2}F_{1}(a,[b,c;y],1)=\frac{\Gamma (c)\Gamma (c-a-b)}{\Gamma (c-a)\Gamma
(c-b)}-\frac{(1-y)^{c-b-a}y^{b}}{B(b,c-b)(c-b-a)}%
~_{2}F_{1}(c-a,1;1+c-b-a;1-y).  \label{2.31}
\end{equation}
\end{proof}

\begin{theorem}
The following result holds true:%
\begin{equation}
_{2}F_{1}(a,\left\{ b,c;y\right\} ,1)=\frac{\Gamma (c)\Gamma (c-a-b)}{\Gamma
(c-a)\Gamma (c-b)}-\frac{(1-y)^{c-b-a}y^{b}}{B(b,c-b)b}%
~_{2}F_{1}(c-a,1;b+1;y)~.  \label{2.34}
\end{equation}
\end{theorem}

\begin{theorem}
The following integral representations hold true:%
\begin{equation}
_{1}F_{1}([a,b;y],x)=\frac{y^{a}}{B(a,b-a)}%
\int_{0}^{1}u^{a-1}(1-uy)^{b-a-1}e^{xuy}du,\text{ \ \ }\func{Re}(b)>\func{Re}%
(a)>0  \label{2.17}
\end{equation}%
and%
\begin{equation}
_{1}F_{1}(\left\{ a,b;y\right\} ,x)=\frac{(1-y)^{b-a}}{B(a,b-a)}%
\int_{0}^{1}u^{b-a-1}(1-u(1-y))^{a-1}e^{(1-u(1-y))x}du,\text{\ \ }\func{Re}%
(b)>\func{Re}(a)>0.  \label{2.18}
\end{equation}
\end{theorem}

\begin{proof}
Replacing the incomplete Pochhammer ratio $[a,b;y]$ in the definition (\ref%
{5.8}) by its integral representation given by (\ref{2.1}), we are led to
the desired result (\ref{2.17}). Formula (\ref{2.18}) can be proved in a
similar way.
\end{proof}

\begin{theorem}
The following integral representation holds true:%
\begin{equation}
\int_{0}^{1}y^{k-1}~_{2}F_{1}(a,[b,c-k;y];x)dy=\frac{1}{k}\left[
_{2}F_{1}(a,b;c-k;x)-\frac{\Gamma \left( c-k\right) \Gamma \left( b+k\right) 
}{\Gamma \left( b\right) \Gamma \left( c\right) }~_{2}F_{1}(a,b+k;c;x)\right]
,~~k\in 
\mathbb{N}
.  \label{4.9}
\end{equation}
\end{theorem}

\begin{proof}
It is known that from the Euler's formula%
\begin{equation*}
_{2}F_{1}(a,b+k;c;x)=\frac{1}{B(b+k,c-b-k)}%
\int_{0}^{1}y^{b+k-1}(1-y)^{c-b-k-1}(1-xy)^{-a}dy,~~k\in 
\mathbb{N}
.
\end{equation*}%
Taking$~u=y^{k}~$and the remaining part as $dv$ and applying the integration
by parts, we get 
\begin{equation*}
_{2}F_{1}(a,b+k;c;x)=\frac{\Gamma \left( b\right) \Gamma \left( c\right) }{%
\Gamma \left( c-k\right) \Gamma \left( b+k\right) }\left[
_{2}F_{1}(a,b;c-k;x)-k\int_{0}^{1}y^{k-1}~_{2}F_{1}(a,[b,c-k;y],x)dy\right] .
\end{equation*}%
By rearranging the terms we get the result.
\end{proof}

\begin{corollary}
Taking $k=1$ in Theorem 7, we get the following result:%
\begin{equation}
\int_{0}^{1}{}_{2}F_{1}(a,[b,c-1;y],x)dy=~_{2}F_{1}(a,b;c-1;x)-\frac{b}{c-1}%
~_{2}F_{1}(a,b+1;c;x).  \label{5.0}
\end{equation}
\end{corollary}

\begin{theorem}
The following integral representation holds true:%
\begin{equation}
\int_{0}^{1}y^{k-1}~_{2}F_{1}(a,[b,c;y],x)dy=\frac{1}{k}\frac{\Gamma \left(
c\right) \Gamma \left( c-b+k\right) }{\Gamma \left( c-b\right) \Gamma \left(
c+k\right) }~_{2}F_{1}(a,b;c+k;x).  \label{5.1}
\end{equation}
\end{theorem}

\begin{proof}
It is known that%
\begin{equation*}
_{2}F_{1}(a,b;c+k;x)=\frac{1}{B(b,c-b+k)}%
\int_{0}^{1}y^{b-1}(1-y)^{c-b+k-1}(1-xy)^{-a}dy.
\end{equation*}%
Taking $u=(1-y)^{k}$ and the rest as $dv$ and using integration by parts, we
get the result.
\end{proof}

\begin{corollary}
Taking $k=1$ in Theorem 9, we get the following result:%
\begin{equation}
_{2}F_{1}(a,b;c+1;x)=\frac{c}{c-b}\int_{0}^{1}\text{ }%
_{2}F_{1}(a,[b,c;y],x)dy.  \label{5.2}
\end{equation}
\end{corollary}

\begin{theorem}
The following derivative formula holds true:%
\begin{equation}
\frac{d^{n}}{dx^{n}}(_{2}F_{1}(a,[b,c;y];x))=\frac{(a)_{n}(b)_{n}}{(c)_{n}}%
~_{2}F_{1}(a+n,[b+n,c+n;y];x).  \label{2.19}
\end{equation}
\end{theorem}

\begin{proof}
Using (\ref{5.11}), differentiating on both sides with respect to $x,$ we
obtain%
\begin{eqnarray*}
\frac{d}{dx}\left( _{2}F_{1}(a,[b,c;y];x)\right) &=&\frac{a}{B(b,c-b)}%
\int_{0}^{y}t^{b}(1-t)^{c-b-1}(1-xt)^{-a-1}dt \\
&=&\frac{a}{B(b,c-b)}%
\int_{0}^{y}t^{(b+1)-1}(1-t)^{(c+1)-(b+1)-1}(1-xt)^{-(a+1)}dt \\
&=&\frac{ab}{c}\frac{1}{B(b+1,c-b)}%
\int_{0}^{y}t^{(b+1)-1}(1-t)^{(c+1)-(b+1)-1}(1-xt)^{-(a+1)}dt \\
&=&\frac{ab}{c}~_{2}F_{1}(a+1,[b+1,c+1;y];x
\end{eqnarray*}%
which is (\ref{2.19}) for $n=1$. The general result follows by the principle
of mathematical induction on $n.$
\end{proof}

\begin{theorem}
The following derivative formula holds true:%
\begin{equation}
\frac{d^{n}}{dx^{n}}(_{1}F_{1}([a,b;y];x))=\frac{(a)_{n}}{(b)_{n}}%
~_{1}F_{1}([a+n,b+n;y];x).  \label{2.20}
\end{equation}
\end{theorem}

\begin{theorem}
We have the following difference formula for $_{2}F_{1}(a,[b,b+h;y];x)$ :\ 
\begin{eqnarray}
&&\frac{b+h-1}{B(b,h)}%
y^{b-1}(1-y)^{h-1}(1-xy)^{-a}=~_{2}F_{1}(a,[b,b+h-1;y];x)+  \notag \\
&&~_{2}F_{1}(a,[b-1,b+h-1;y];x)-ax(b+h-1)_{~2}F_{1}(a+1,[b,b+h;y];x).
\label{5.12}
\end{eqnarray}
\end{theorem}

\begin{proof}
Recalling that the Mellin transform operator is defined by%
\begin{equation*}
\mathfrak{M}\left\{ f(t):s\right\} :=\int_{0}^{\infty }t^{s-1}f(t)dt,~\func{%
Re}(s)>0,
\end{equation*}%
we observe that $_{2}F_{1}(a,[b,b+h;y];x)$ is the Mellin transform of the
function%
\begin{equation*}
f(t:x;y,a;h)=H(y-t)(1-t)^{h-1}(1-xt)^{-a},
\end{equation*}%
where%
\begin{equation*}
H(t)=\left\{ 
\begin{array}{l}
1\text{ \ \ if \ \ }t>0 \\ 
0\text{ \ \ if \ \ }t<0%
\end{array}%
\right. ,
\end{equation*}%
is the Heaviside unit function. Observing the fact that 
\begin{equation}
_{2}F_{1}(a,[b,b+h;y];x):=\frac{\mathfrak{M}\left\{ f(t:x;y,a;h):b\right\} }{%
B(b,h)},  \label{5.13}
\end{equation}%
we can write that 
\begin{eqnarray}
\frac{\partial }{\partial t}(f(t
&:&x;y,a;h))=-[(y-t)(1-t)^{h-1}(1-xt)^{-a}+(h-1)H(y-t)(1-t)^{h-2}(1-xt)^{-a}]
\label{5.14} \\
&&\text{ \ \ \ \ \ \ \ \ \ \ \ \ \ \ \ \ \ \ }%
+ax(1-xt)^{-a-1}H(y-t)(1-t)^{h-1},  \notag
\end{eqnarray}%
where $\frac{\partial }{\partial t}(H(t))=\delta (t-t_{0}),$ 
\begin{equation*}
\delta (t-t_{0})=\left\{ 
\begin{array}{l}
\infty \text{ \ \ if \ \ }t=t_{0} \\ 
0\text{ \ \ \ if \ \ }t\neq t_{0}%
\end{array}%
\right. ,
\end{equation*}%
is the Dirac delta function. Applying Mellin transform on both sides (\ref%
{5.14}) and using (\ref{5.13}) and the fact that%
\begin{equation*}
\mathfrak{M}\left\{ f^{\prime }(t):x\right\} =(1-x)\mathfrak{M}\left\{
f(t):x-1\right\} ,
\end{equation*}%
we have%
\begin{eqnarray*}
&&\frac{b+h-1}{B(b,h)}%
y^{b-1}(1-y)^{h-1}(1-xy)^{-a}=~_{2}F_{1}(a,[b,b+h-1;y];x) \\
&&+~_{2}F_{1}(a,[b-1,b+h-1;y];x)-ax(b+h-1)~_{2}F_{1}(a+1,[b,b+h;y];x).
\end{eqnarray*}%
This completes the proof.
\end{proof}

In the following theorems, we give transformation formulas:

\begin{theorem}
The following transformation formula holds true:%
\begin{equation}
_{2}F_{1}(a,[\beta ,\gamma ;y];z)=(1-z)^{-a}{}~_{2}F_{1}(a,\left\{ \gamma
-\beta ,\gamma ;1-y\right\} ;\frac{z}{z-1}),~\left\vert \arg
(1-z)\right\vert <\pi .  \label{2.21}
\end{equation}
\end{theorem}

\begin{proof}
Using (\ref{5.11}), we obtain 
\begin{equation}
_{2}F_{1}(a,[\beta ,\gamma ;y];z)=\frac{\left( 1-z\right) ^{-a}}{B(\beta
,\gamma -\beta )}\int_{1-y}^{1}(1-s)^{\beta -1}s^{\gamma -\beta -1}\left( 1-%
\frac{z}{z-1}s\right) ^{-a}ds.  \label{2.22}
\end{equation}%
The substitution $s=1-t$ in (\ref{2.22}) leads to%
\begin{eqnarray*}
_{2}F_{1}(a,[\beta ,\gamma ;y];z) &=&\frac{\left( 1-z\right) ^{-a}}{B(\beta
,\gamma -\beta )}\int_{0}^{y}t^{\beta -1}(1-t)^{\gamma -\beta -1}\left( 1-%
\frac{z(1-t)}{z-1}\right) ^{-a}dt \\
&=&(1-z)^{-a}{}~_{2}F_{1}(a,\left\{ \gamma -\beta ,\gamma ;1-y\right\} ;%
\frac{z}{z-1}).
\end{eqnarray*}
\end{proof}

\begin{theorem}
The following transformation formula holds true:%
\begin{equation}
_{2}F_{1}(a,\left\{ \beta ,\gamma ;y\right\}
;z)=(1-z)^{-a}{}~_{2}F_{1}(a,[\gamma -\beta ,\gamma ;1-y];\frac{z}{z-1}%
),~~~\left\vert \arg (1-z)\right\vert <\pi .  \label{2.23}
\end{equation}
\end{theorem}

\begin{theorem}
The following transformation formulas hold true:%
\begin{equation}
_{1}F_{1}(\left\{ \alpha ,\beta ;1-y\right\} ;z)=e^{z}{}~_{1}F_{1}(\left[
\beta -\alpha ,\beta ;y\right] ;-z)  \label{2.24}
\end{equation}%
and 
\begin{equation}
_{1}F_{1}(\left[ \alpha ,\beta ;y\right] ;z)=e^{z}{}~_{1}F_{1}\left( \left\{
\beta -\alpha ,\beta ;1-y\right\} ;-z\right) .  \label{2.25}
\end{equation}
\end{theorem}

\begin{proof}
The proofs of (\ref{2.24}) and (\ref{2.25}) are direct consequences of
Theorem 6.
\end{proof}

\section{The incomplete Appell's functions}

\bigskip In this section, we introduce the incomplete Appell's functions $%
F_{1}[a,b,c;d;x,z;y]$, $F_{1}\{a,b,c;d;x,z;y\}$, $F_{2}[a,b,c;d,e;x,z;y]$
and $F_{2}\{a,b,c;d,e;x,z;y\}$ by

\begin{equation}
F_{1}[a,b,c;d;x,z;y]:=\sum_{m,n=0}^{\infty }[a,d;y]_{m+n}(b)_{m}(c)_{n}\frac{%
x^{m}}{m!}\frac{z^{n}}{n!},\text{ \ \ }\max \{\left\vert x\right\vert
,\left\vert z\right\vert \}<1  \label{3.4}
\end{equation}%
and%
\begin{equation}
F_{1}\{a,b,c;d;x,z;y\}:=\sum_{m,n=0}^{\infty }\{a,d;y\}_{m+n}(b)_{m}(c)_{n}%
\frac{x^{m}}{m!}\frac{z^{n}}{n!},\text{ \ \ }\max \{\left\vert x\right\vert
,\left\vert z\right\vert \}<1  \label{3.13}
\end{equation}%
and%
\begin{equation}
F_{2}[a,b,c;d,e;x,z;y]:=\sum_{m,n=0}^{\infty }(a)_{m+n}[b,d;y]_{m}[c,e;y]_{n}%
\frac{x^{m}}{m!}\frac{z^{n}}{n!},\ \ \left\vert x\right\vert +\left\vert
z\right\vert <1  \label{3.5}
\end{equation}%
and%
\begin{equation}
F_{2}\{a,b,c;d,e;x,z;y\}:=\sum_{m,n=0}^{\infty
}(a)_{m+n}\{b,d;y\}_{m}\{c,e;y\}_{n}\frac{x^{m}}{m!}\frac{z^{n}}{n!},\text{
\ \ }\left\vert x\right\vert +\left\vert z\right\vert <1.  \label{3.14}
\end{equation}%
We proceed by obtaining the integral representations of the functions $%
F_{1}[a,b,c;d;x,z;y]$, $F_{1}\{a,b,c;d;x,z;y\}$, $F_{2}[a,b,c;d,e;x,z;y]$
and $F_{2}\{a,b,c;d,e;x,z;y\}.$

\begin{theorem}
For the incomplete Appell's functions $F_{1}[a,b,c;d;x,z;y]$ and $%
F_{1}\{a,b,c;d;x,z;y\},$ we have the following integral representation:%
\begin{eqnarray}
F_{1}[a,b,c;d;x,z;y] &=&\frac{y^{a}}{B(a,d-a)}%
\int_{0}^{1}u^{a-1}(1-uy)^{d-a-1}(1-xuy)^{-b}(1-zuy)^{-c}du,\text{ \ \ }
\label{3.6} \\
\func{Re}(d) &>&0,~\func{Re}(a)>0,~\func{Re}(b)>0,~\func{Re}%
(c)>0,~\left\vert \arg \left( 1-x\right) \right\vert <\pi ,~\left\vert \arg
\left( 1-z\right) \right\vert <\pi .  \notag
\end{eqnarray}%
and%
\begin{eqnarray}
F_{1}\{a,b,c;d;x,z;y\} &=&\frac{(1-y)^{d-a}}{B(a,d-a)}  \notag \\
&&\times
\int_{0}^{1}u^{d-a-1}(1-u(1-y))^{a-1}(1-x(1-u(1-y)))^{-b}(1-z(1-u(1-y)))^{-c}du,
\notag \\
\func{Re}(d) &>&0,~\func{Re}(a)>0,~\func{Re}(b)>0,~\func{Re}%
(c)>0,~\left\vert \arg \left( 1-x\right) \right\vert <\pi ,~\left\vert \arg
\left( 1-z\right) \right\vert <\pi .  \label{3.15}
\end{eqnarray}
\end{theorem}

\begin{proof}
Replacing the integral representation for incomplete beta function which is
given by (\ref{2.1}), we find that%
\begin{equation*}
F_{1}[a,b,c;d;x,z;y]=\frac{1}{B(a,d-a)}%
\int_{0}^{y}t^{a-1}(1-t)^{d-a-1}(1-xt)^{-b}(1-zt)^{-c}dt,
\end{equation*}%
which can be written as 
\begin{equation*}
F_{1}[a,b,c;d;x,z;y]=\frac{y^{a}}{B(a,d-a)}%
\int_{0}^{1}u^{a-1}(1-uy)^{d-a-1}(1-xuy)^{-b}(1-zuy)^{-c}du.
\end{equation*}%
Whence the result. Formula (\ref{3.15}) can be proved in a similar way.
\end{proof}

\begin{theorem}
For the incomplete Appell's functions $F_{2}[a,b,c;d,e;x,z;y]$ and $%
F_{2}\{a,b,c;d,e;x,z;y\},$ we have the following integral representation:%
\begin{eqnarray}
&&F_{2}[a,b,c;d,e;x,z;y]=\frac{y^{b+c}}{B(b,d-b)B(c,e-c)}  \notag \\
&&\times
\int_{0}^{1}%
\int_{0}^{1}u^{b-1}(1-uy)^{d-b-1}v^{c-1}(1-vy)^{e-c-1}(1-xuy-zvy)^{-a}dudv, 
\notag \\
\text{ \ \ \ \ \ \ \ \ }\func{Re}(d) &>&\func{Re}(a)>\func{Re}(b)>\func{Re}%
(c)>\func{Re}(m)>0,~\left\vert \arg \left( 1-x-z\right) \right\vert <\pi .~
\label{3.9}
\end{eqnarray}%
and%
\begin{eqnarray}
&&F_{2}\{a,b,c;d,e;x,z;y\}  \notag \\
&=&\frac{(1-y)^{d-b+e-c}}{B(b,d-b)B(c,e-c)}\int_{0}^{1}%
\int_{0}^{1}u^{d-b-1}(1-u(1-y))^{b-1}v^{e-c-1}(1-v(1-y))^{c-1}  \notag \\
&&(1-x(1-u(1-y))-z(1-v(1-y)))^{-a}dudv,~  \notag \\
\func{Re}(d) &>&0,~\func{Re}(a)>0,~\func{Re}(b)>0,~\func{Re}(c)>0,~\func{Re}%
(e)>0,~\left\vert \arg \left( 1-x-z\right) \right\vert <\pi .  \label{3.16}
\end{eqnarray}
\end{theorem}

\begin{proof}
Replacing the integral representation for incomplete beta function which is
given by (\ref{2.1}), we get 
\begin{eqnarray*}
F_{2}[a,b,c;d,e;x,z;y] &=&\frac{1}{B(b,d-b)B(c,e-c)} \\
&&\times \sum_{m,n=0}^{\infty
}\int_{0}^{y}%
\int_{0}^{y}(a)_{m+n}t^{b+m-1}(1-t)^{d-b-1}s^{c+n-1}(1-s)^{e-c-1}\frac{x^{m}%
}{m!}\frac{z^{n}}{n!}dtds.
\end{eqnarray*}%
Considering the fact that the series involved are uniformly convergent and
we have a right to interchange the order of summation and integration, we
get 
\begin{eqnarray*}
F_{2}[a,b,c;d,e;x,z;y] &=&\frac{1}{B(b,d-b)B(c,e-c)} \\
&&\times
\int_{0}^{y}%
\int_{0}^{y}t^{b-1}(1-t)^{d-b-1}s^{c-1}(1-s)^{e-c-1}(1-xt-zs)^{-a}dtds, \\
&=&\frac{y^{b+c}}{B(b,d-b)B(c,e-c)} \\
&&\times
\int_{0}^{1}%
\int_{0}^{1}u^{b-1}(1-uy)^{d-b-1}v^{c-1}(1-vy)^{e-c-1}(1-xuy-zvy)^{-a}dudv.
\end{eqnarray*}%
Formula (\ref{3.16}) can be proved in a similar way.
\end{proof}

\section{Incomplete Riemann-Liouville fractional derivative operator}

\bigskip In this section, we introduce and investigate the incomplete
Riemann-Liouville fractional derivative operators. The Riemann-Liouville
fractional derivative of order $\mu $ is defined by%
\begin{equation}
D_{z}^{\mu }\{f(z)\}:=\frac{1}{\Gamma \left( -\mu \right) }%
\int_{0}^{z}f(t)(z-t)^{-\mu -1}dt,\text{ \ \ }\func{Re}(\mu )<0.  \label{4.1}
\end{equation}

Now, we define the incomplete Riemann-Liouville fractional derivative
operators $D_{z}^{\mu }[f(z);y]$ and $D_{z}^{\mu }\{f(z);y\}$ by 
\begin{eqnarray}
D_{z}^{\mu }[f(z);y] &:&=\frac{z^{-\mu }}{\Gamma \left( -\mu \right) }%
\int_{0}^{y}f(uz)(1-u)^{-\mu -1}du  \label{4.2} \\
&:&=\frac{z^{-\mu }y}{\Gamma \left( -\mu \right) }%
\int_{0}^{1}f(ywz)(1-wy)^{-\mu -1}dw,\text{ }\func{Re}(\mu )<0.  \notag
\end{eqnarray}

and its counterpart is by%
\begin{eqnarray}
D_{z}^{\mu }\{f(z);y\} &:&=\frac{z^{-\mu }}{\Gamma \left( -\mu \right) }%
\int_{y}^{1}f(uz)(1-u)^{-\mu -1}du  \label{4.3} \\
&:&=\frac{z^{-\mu }}{\Gamma \left( -\mu \right) }\int_{0}^{1-y}f((1-t)z)t^{-%
\mu -1}dt,\text{ \ }\func{Re}(\mu )<0.  \notag
\end{eqnarray}

We start our investigation by calculating the incomplete fractional
derivatives of some elementary functions.

\begin{theorem}
Let $\func{Re}(\lambda )>-1,~\func{Re}(\mu )<0.$ Then%
\begin{equation}
D_{z}^{\mu }[z^{\lambda };y]=\frac{B_{y}(\lambda +1,-\mu )}{\Gamma \left(
-\mu \right) }z^{\lambda -\mu }.  \label{4.4}
\end{equation}
\end{theorem}

\begin{proof}
Using (\ref{4.2}) and (\ref{2.1}), we get%
\begin{eqnarray*}
D_{z}^{\mu }[z^{\lambda };y] &=&\frac{z^{-\mu }}{\Gamma \left( -\mu \right) }%
\int_{0}^{y}(uz)^{\lambda }(1-u)^{-\mu -1}du \\
&=&\frac{B_{y}(\lambda +1,-\mu )}{\Gamma \left( -\mu \right) }z^{\lambda
-\mu }.
\end{eqnarray*}%
Whence the result.
\end{proof}

\begin{theorem}
Let $\func{Re}(\lambda )>-1,~\func{Re}(\mu )<0.$ Then%
\begin{equation}
D_{z}^{\mu }\{z^{\lambda };y\}=\frac{B_{1-y}(-\mu ,\lambda +1)}{\Gamma
\left( -\mu \right) }z^{-\mu +\lambda }.  \label{4.5}
\end{equation}
\end{theorem}

\begin{theorem}
Let $\func{Re}(\lambda )>0,~\func{Re}(\alpha )>0,~\func{Re}(\mu )<0$ and$%
~\left\vert z\right\vert <1.~$Then%
\begin{equation}
D_{z}^{\lambda -\mu }[z^{\lambda -1}(1-z)^{-\alpha };y]=\frac{\Gamma \left(
\lambda \right) }{\Gamma \left( \mu \right) }z^{\mu -1}~_{2}F_{1}(\alpha ,%
\left[ \lambda ,\mu ;y\right] ;z),  \label{4.6}
\end{equation}%
and%
\begin{equation}
D_{z}^{\lambda -\mu }\{z^{\lambda -1}(1-z)^{-\alpha };y\}=\frac{\Gamma
\left( \lambda \right) }{\Gamma \left( \mu \right) }z^{\mu
-1}~_{2}F_{1}(\alpha ,\{\lambda ,\mu ;y\};z).  \label{5.15}
\end{equation}
\end{theorem}

\begin{proof}
Direct calculations yield%
\begin{eqnarray*}
D_{z}^{\lambda -\mu }[z^{\lambda -1}(1-z)^{-\alpha };y] &=&\frac{z^{\mu
-\lambda }}{\Gamma \left( \mu -\lambda \right) }\int_{0}^{y}(uz)^{\lambda
-1}(1-uz)^{-\alpha }(1-u)^{\mu -\lambda -1}du \\
&=&\frac{z^{\mu -\lambda }y}{\Gamma \left( \mu -\lambda \right) }%
\int_{0}^{1}(yz)^{\lambda -1}w^{\lambda -1}(1-ywz)^{-\alpha }(1-wy)^{\mu
-\lambda -1}dw \\
&=&\frac{z^{\mu -1}y^{\lambda }}{\Gamma \left( \mu -\lambda \right) }%
\int_{0}^{1}w^{\lambda -1}(1-ywz)^{-\alpha }(1-wy)^{\mu -\lambda -1}dw.
\end{eqnarray*}%
By (\ref{2.13}), we can write%
\begin{eqnarray*}
D_{z}^{\lambda -\mu }[z^{\lambda -1}(1-z)^{-\alpha };y] &=&\frac{z^{\mu -1}}{%
\Gamma \left( \mu -\lambda \right) }B(\lambda ,\mu -\lambda
)_{2}F_{1}(\alpha ,\left[ \lambda ,\mu ;y\right] ;z) \\
&=&\frac{\Gamma \left( \lambda \right) }{\Gamma \left( \mu \right) }z^{\mu
-1}~_{2}F_{1}(\alpha ,\left[ \lambda ,\mu ;y\right] ;z).
\end{eqnarray*}%
Hence the proof is completed. Formula (\ref{5.15}) can be proved in a
similar way.
\end{proof}

\begin{theorem}
Let $\func{Re}(\lambda )>\func{Re}(\mu )>0,~\func{Re}(\alpha )>0,~\func{Re}%
(\beta )>0$ ; $\left\vert az\right\vert <1~$and $\left\vert bz\right\vert
<1. $ Then%
\begin{equation}
D_{z}^{\lambda -\mu }[z^{\lambda -1}(1-az)^{-\alpha }(1-bz)^{-\beta };y]=%
\frac{\Gamma \left( \lambda \right) }{\Gamma \left( \mu \right) }z^{\mu
-1}F_{1}[\lambda ,\alpha ,\beta ;\mu ;az,bz;y],  \label{4.7}
\end{equation}%
and%
\begin{equation}
D_{z}^{\lambda -\mu }\{z^{\lambda -1}(1-az)^{-\alpha }(1-bz)^{-\beta };y\}=%
\frac{\Gamma \left( \lambda \right) }{\Gamma \left( \mu \right) }z^{\mu
-1}F_{1}\{\lambda ,\alpha ,\beta ;\mu ;az,bz;y\}.  \label{5.16}
\end{equation}
\end{theorem}

\begin{proof}
We have 
\begin{eqnarray*}
&&D_{z}^{\lambda -\mu }[z^{\lambda -1}(1-az)^{-\alpha }(1-bz)^{-\beta };y] \\
&=&\frac{z^{\mu -\lambda }}{\Gamma \left( \mu -\lambda \right) }%
\int_{0}^{y}(uz)^{\lambda -1}(1-auz)^{-\alpha }(1-buz)^{-\beta }(1-u)^{\mu
-\lambda -1}du \\
&=&\frac{z^{\mu -\lambda }y}{\Gamma \left( \mu -\lambda \right) }%
\int_{0}^{1}(yw)^{\lambda -1}(z)^{\lambda -1}(1-aywz)^{-\alpha
}(1-bywz)^{-\beta }(1-wy)^{\mu -\lambda -1}dw \\
&=&\frac{z^{\mu -1}y^{\lambda }}{\Gamma \left( \mu -\lambda \right) }%
\int_{0}^{1}w^{\lambda -1}(1-aywz)^{-\alpha }(1-bywz)^{-\beta }(1-wy)^{\mu
-\lambda -1}dw.
\end{eqnarray*}%
By (\ref{3.6}), we can write%
\begin{eqnarray*}
D_{z}^{\lambda -\mu }[z^{\lambda -1}(1-az)^{-\alpha }(1-bz)^{-\beta };y] &=&%
\frac{z^{\mu -1}}{\Gamma \left( \mu -\lambda \right) }B(\lambda ,\mu
-\lambda )F_{1}[\lambda ,\alpha ,\beta ;\mu ;az,bz;y] \\
&=&\frac{\Gamma \left( \lambda \right) }{\Gamma \left( \mu \right) }z^{\mu
-1}F_{1}[\lambda ,\alpha ,\beta ;\mu ;az,bz;y].
\end{eqnarray*}%
Whence the result. Formula (\ref{5.16}), can be proved in a similar way.
\end{proof}

\begin{theorem}
Let $\func{Re}(\lambda )>\func{Re}(\mu )>0,\func{Re}(\alpha )>0,\func{Re}%
(\beta )>0,\func{Re}(\gamma )>0$ ; $\left\vert \frac{t}{1-z}\right\vert <1$%
and $\left\vert t\right\vert +\left\vert z\right\vert <1$ we have%
\begin{equation}
D_{z}^{\lambda -\mu }[z^{\lambda -1}(1-z)^{-\alpha }~_{2}F_{1}(\alpha .\left[
\beta ,\gamma ;y\right] ;\frac{t}{1-z});y]=\frac{\Gamma \left( \lambda
\right) }{\Gamma \left( \mu \right) }z^{\mu -1}F_{2}[\alpha ,\beta ,\lambda
;\gamma ,\mu ;t,z;y],  \label{4.8}
\end{equation}%
and%
\begin{equation}
D_{z}^{\lambda -\mu }\{z^{\lambda -1}(1-z)^{-\alpha }~_{2}F_{1}(\alpha . 
\left[ \beta ,\gamma ;y\right] ;\frac{t}{1-z});y\}=\frac{\Gamma \left(
\lambda \right) }{\Gamma \left( \mu \right) }z^{\mu -1}F_{2}\{\alpha ,\beta
,\lambda ;\gamma ,\mu ;t,z;y\}.  \label{5.17}
\end{equation}
\end{theorem}

\begin{proof}
Using Theorem 19 and (\ref{3.5}), we get 
\begin{eqnarray*}
&&D_{z}^{\lambda -\mu }[z^{\lambda -1}(1-z)^{-\alpha }~_{2}F_{1}(\alpha . 
\left[ \beta ,\gamma ;y\right] ;\frac{t}{1-z});y] \\
&=&~D_{z}^{\lambda -\mu }[z^{\lambda -1}(1-z)^{-\alpha }\frac{1}{B(\beta
,\gamma -\beta )}\sum_{n=0}^{\infty }\frac{\left( \alpha \right)
_{n}B_{y}(\beta +n,\gamma -\beta )}{n!}\left( \frac{t}{1-z}\right) ^{n};y] \\
&=&\frac{1}{B(\beta ,\gamma -\beta )}D_{z}^{\lambda -\mu }[z^{\lambda
-1}\sum_{n=0}^{\infty }\left( \alpha \right) _{n}B_{y}(\beta +n,\gamma
-\beta )\frac{t^{n}}{n!}(1-z)^{-\alpha -n};y] \\
&=&\frac{1}{B(\beta ,\gamma -\beta )}\sum_{m,n=0}^{\infty }B_{y}(\beta
+n,\gamma -\beta )\frac{t^{n}}{n!}\frac{\left( \alpha \right) _{n}(\alpha
+n)_{m}}{m!}D_{z}^{\lambda -\mu }[z^{\lambda -1+m};y] \\
&=&\frac{1}{B(\beta ,\gamma -\beta )}\sum_{m,n=0}^{\infty }B_{y}(\beta
+n,\gamma -\beta )\frac{t^{n}}{n!}\frac{\left( \alpha \right) _{n+m}}{m!}%
\frac{B_{y}(\lambda +m,\mu -\lambda )}{\Gamma (\mu -\lambda )}z^{\mu +m-1} \\
&=&\frac{\Gamma \left( \lambda \right) }{\Gamma \left( \mu \right) }z^{\mu
-1}F_{2}[\alpha ,\beta ,\lambda ;\gamma ,\mu ;t,z;y].
\end{eqnarray*}%
Hence proof is completed. Formula (\ref{5.17}), can be proved in a similar
way.
\end{proof}

\section{\protect\bigskip Generating Functions}

Now, we obtain linear and bilinear generating relations for the incomplete
hypergeometric functions $_{2}F_{1}(a,\left[ b,c;y\right] ;x)$ by following
the methods described in \cite{4}. We start with the following theorem:

\begin{theorem}
For the incomplete hypergeometric functions we have%
\begin{equation}
\sum_{n=0}^{\infty }\frac{(\lambda )_{n}}{n!}~_{2}F_{1}(\lambda +n,\left[
\alpha ,\beta ;y\right] ;z)t^{n}=(1-t)^{-\lambda }~_{2}F_{1}(\lambda ,\left[
\alpha ,\beta ;y\right] ;\frac{z}{1-t})  \label{5.4}
\end{equation}%
and%
\begin{equation}
\sum_{n=0}^{\infty }\frac{(\lambda )_{n}}{n!}~_{2}F_{1}(\lambda +n,\{\alpha
,\beta ;y\};z)t^{n}=(1-t)^{-\lambda }~_{2}F_{1}(\lambda ,\{\alpha ,\beta
;y\};\frac{z}{1-t})  \label{5.18}
\end{equation}%
where $\left\vert z\right\vert <\min \{1,\left\vert 1-t\right\vert \}$ and~ $%
\func{Re}(\lambda )>0$,$~\func{Re}(\beta )>\func{Re}(\alpha )>0.$
\end{theorem}

\begin{proof}
Considering the elementary identity 
\begin{equation*}
\lbrack (1-z)-t]^{-\lambda }=(1-t)^{-\lambda }\left[ 1-\frac{z}{1-t}\right]
^{-\lambda }
\end{equation*}%
and expanding the left hand side, we have for $\left\vert t\right\vert
<\left\vert 1-z\right\vert $ that 
\begin{equation*}
(1-z)^{-\lambda }\sum_{n=0}^{\infty }\frac{(\lambda )_{n}}{n!}\left( \frac{t%
}{1-z}\right) ^{n}=(1-t)^{-\lambda }\left[ 1-\frac{z}{1-t}\right] ^{-\lambda
}.
\end{equation*}%
Now, multiplying both sides of the above equality by $z^{\alpha -1}$ and
applying the incomplete fractional derivative operator $D_{z}^{\alpha -\beta
}[f(z);y]$ on both sides , we can write 
\begin{equation*}
D_{z}^{\alpha -\beta }\left[ \sum_{n=0}^{\infty }\frac{(\lambda )_{n}}{n!}%
(1-z)^{-\lambda }\left( \frac{t}{1-z}\right) ^{n}z^{\alpha -1};y\right]
=(1-t)^{-\lambda }D_{z}^{\alpha -\beta }\left[ z^{\alpha -1}\left[ 1-\frac{z%
}{1-t}\right] ^{-\lambda };y\right] .
\end{equation*}%
Interchanging the order, which is valid for $\func{Re}(\alpha )>0$ and $%
\left\vert t\right\vert <\left\vert 1-z\right\vert ,$ we get%
\begin{equation*}
\sum_{n=0}^{\infty }\frac{(\lambda )_{n}}{n!}D_{z}^{\alpha -\beta }\left[
z^{\alpha -1}(1-z)^{-\lambda -n};y\right] t^{n}=(1-t)^{-\lambda
}D_{z}^{\alpha -\beta }\left[ z^{\alpha -1}\left[ 1-\frac{z}{1-t}\right]
^{-\lambda };y\right] .
\end{equation*}%
Using Theorem 21, we get the desired result. Formula (\ref{5.18}), can be
proved in a similar way.
\end{proof}

The following theorem gives another linear generating relation for the
incomplete hypergeometric functions.

\begin{theorem}
For the incomplete hypergeometric functions we have%
\begin{equation}
\sum_{n=0}^{\infty }\frac{(\lambda )_{n}}{n!}~_{2}F_{1}(\rho -n,\left[
\alpha ,\beta ;y\right] ;z)t^{n}=(1-t)^{-\lambda }F_{1}[\alpha ,\rho
,\lambda ;\beta ;z;\frac{-zt}{1-t};y]  \label{5.5}
\end{equation}%
and%
\begin{equation}
\sum_{n=0}^{\infty }\frac{(\lambda )_{n}}{n!}~_{2}F_{1}(\rho -n,\left\{
\alpha ,\beta ;y\right\} ;z)t^{n}=(1-t)^{-\lambda }F_{1}\{\alpha ,\rho
,\lambda ;\beta ;z;\frac{-zt}{1-t};y\}  \label{5.19}
\end{equation}%
where $\func{Re}(\lambda )>0,~\func{Re}(\rho )>0$,$~\func{Re}(\beta )>\func{%
Re}(\alpha )>0;$ $\left\vert t\right\vert <\frac{1}{1+\left\vert
z\right\vert }.$
\end{theorem}

\begin{proof}
Considering 
\begin{equation*}
\lbrack 1-(1-z)t]^{-\lambda }=(1-t)^{-\lambda }\left[ 1+\frac{zt}{1-t}\right]
^{-\lambda }
\end{equation*}%
and expanding the left hand side, we have for $\left\vert t\right\vert
<\left\vert 1-z\right\vert $ that%
\begin{equation*}
\sum_{n=0}^{\infty }\frac{(\lambda )_{n}}{n!}(1-z)^{n}t^{n}=(1-t)^{-\lambda }%
\left[ 1-\frac{-zt}{1-t}\right] ^{-\lambda }.
\end{equation*}%
Now, multiplying both sides of the above equality by $z^{\alpha
-1}(1-z)^{-\rho }$ and applying the fractional derivative operator $%
D_{z}^{\alpha -\beta }[f(z);y]$ on both sides, we get 
\begin{equation*}
D_{z}^{\alpha -\beta }\left[ \sum_{n=0}^{\infty }\frac{(\lambda )_{n}}{n!}%
z^{\alpha -1}(1-z)^{-\rho +n}t^{n};y\right] =(1-t)^{-\lambda }D_{z}^{\alpha
-\beta }\left[ z^{\alpha -1}(1-z)^{-\rho }\left[ 1-\frac{-zt}{1-t}\right]
^{-\lambda };y\right] .
\end{equation*}%
Interchanging the order, which is valid for $\func{Re}(\alpha )>0$ and $%
\left\vert zt\right\vert <\left\vert 1-t\right\vert ,$ we get 
\begin{equation*}
\sum_{n=0}^{\infty }\frac{(\lambda )_{n}}{n!}D_{z}^{\alpha -\beta }\left[
z^{\alpha -1}(1-z)^{-(\rho -n)};y\right] t^{n}=(1-t)^{-\lambda
}D_{z}^{\alpha -\beta }\left[ z^{\alpha -1}(1-z)^{-\rho }\left[ 1-\frac{-zt}{%
1-t}\right] ^{-\lambda };y\right] .
\end{equation*}%
Using Theorem 21 and 22, we get the desired result. Generating relation (\ref%
{5.19}), can be proved in a similar way.
\end{proof}

Finally we have the following bilinear generating relation for the
incomplete hypergeometric functions.

\begin{theorem}
For the incomplete hypergeometric functions we have%
\begin{equation}
\sum_{n=0}^{\infty }\frac{(\lambda )_{n}}{n!}~_{2}F_{1}(\gamma ,\left[
-n,\delta ;y\right] ;x)~_{2}F_{1}(\gamma ,\left[ \lambda +n,\beta ;y\right]
;z)t^{n}=(1-t)^{-\lambda }F_{2}[\lambda ,\alpha ,\gamma ;\beta ,\delta ;%
\frac{z}{1-t};\frac{-xt}{1-t};y]  \label{5.6}
\end{equation}%
and%
\begin{equation}
\sum_{n=0}^{\infty }\frac{(\lambda )_{n}}{n!}~_{2}F_{1}(\gamma ,\left\{
-n,\delta ;y\right\} ;x)~_{2}F_{1}(\gamma ,\{\lambda +n,\beta
;y\};z)t^{n}=(1-t)^{-\lambda }F_{2}\{\lambda ,\alpha ,\gamma ;\beta ,\delta ;%
\frac{z}{1-t};\frac{-xt}{1-t};y\}  \label{5.21}
\end{equation}%
where $\func{Re}(\lambda )>0,~\func{Re}(\gamma )>0$,$~\func{Re}(\beta )>0,~%
\func{Re}(\delta )>0,~\func{Re}(\alpha )>0;$ $\left\vert t\right\vert <\frac{%
1-\left\vert z\right\vert }{1+\left\vert x\right\vert }$ and $\left\vert
z\right\vert <1.$

\begin{proof}
Replacing $t~$by$~(1-x)t$ in (\ref{5.4}), multiplying the resulting equality
by $x^{\gamma -1}$ and then applying the incomplete fractional derivative
operator $D_{x}^{\gamma -\delta }[f(x);y],$ we get 
\begin{eqnarray*}
&&D_{x}^{\gamma -\delta }\left[ \sum_{n=0}^{\infty }\frac{(\lambda )_{n}}{n!}%
x^{\gamma -1}~_{2}F_{1}(\lambda +n,\left[ \alpha ,\beta ;y\right]
;z)(1-x)^{n}t^{n};y\right] \\
&=&D_{x}^{\gamma -\delta }\left[ (1-(1-x)t)^{-\lambda }x^{\gamma
-1}~_{2}F_{1}(\lambda ,\left[ \alpha ,\beta ;y\right] ;\frac{z}{1-(1-x)t});y%
\right] .
\end{eqnarray*}%
Interchanging the order, which is valid for $\left\vert z\right\vert
<1,\left\vert \frac{1-x}{1-z}t\right\vert <1$ and $\left\vert \frac{z}{1-t}%
\right\vert +\left\vert \frac{xt}{1-t}\right\vert <1,$ we can write that 
\begin{eqnarray*}
&&\sum_{n=0}^{\infty }\frac{(\lambda )_{n}}{n!}D_{x}^{\gamma -\delta }\left[
x^{\gamma -1}(1-x)^{n};y\right] ~_{2}F_{1}(\lambda +n,\left[ \alpha ,\beta ;y%
\right] ;z) \\
&=&(1-t)^{-\lambda }D_{x}^{\gamma -\delta }\left[ x^{\gamma -1}(1-\frac{-xt}{%
1-t})~_{2}F_{1}(\lambda ,\left[ \alpha ,\beta ;y\right] ;\frac{\frac{z}{1-t}%
}{1-\frac{-xt}{1-t}});y\right] .
\end{eqnarray*}%
Using Theorems 21 and 23, we get (\ref{5.6}). Generating relation (\ref{5.21}%
), can be proved in a similar way.
\end{proof}
\end{theorem}

\section{\protect\bigskip Conclusion\ }

Incomplete Pochhammer ratios are defined\ in (\ref{2.4}) and (\ref{2.5}) by
using the incomplete beta functions. Several properties of these functions
are obtained. Incomplete hypergeometric functions are introduced with the
help of these incomplete Pochhammer ratios and certain properties such as
integral representations, derivative formulas, transformation formulas and
recurrence relation are investigated. Furthermore, incomplete
Riemann-Liouville fractional derivative operators are defined. The
incomplete Riemann-Liouville fractional derivatives for the some elementary
functions are given. Linear and bilinear generating relations for incomplete
hypergeometric functions are obtained.

\end{document}